\documentclass[12pt, a4paper]{amsart}
\def\doctype{}

\usepackage{latexsym,amssymb,amsthm}
\usepackage{color}
\usepackage{dsfont}
\usepackage{fancyhdr}
\usepackage{nicematrix}
\usepackage{tikz}
\usepackage{hyperref}
\usepackage{float}
\usepackage{enumitem}

\renewcommand\S{\mathrm{S}}

\newcommand{\comment}[1]{}

\newcommand{\ind}{\hspace{-0.07cm}\uparrow^G}
\newcommand{\res}{\hspace{-0.07cm}\downarrow_H}

\def\cn{\mathord{{\!\:{:}\:\!}}}

\fancypagestyle{titlepage}{
\fancyhead[R]{\doctype}
\fancyhead[CL]{}
\cfoot{\vspace{5pt} \thepage}
}

\let\oldsection\section
\newcommand\boldsection[1]{\oldsection{\bf #1}}
\newcommand\starsection[1]{\oldsection*{\bf #1}}
\makeatletter
\renewcommand\section{\@ifstar\starsection\boldsection}
\makeatother

\newtheoremstyle{algorithm}
  {12pt}		  
  {0pt}  
  {\tt}  
  {\parindent}     
  {\bf}  
  {. }    
  {\newline}    
  {}     
\theoremstyle{algorithm}

\newtheoremstyle{theorem}
  {12pt}		  
  {0pt}  
  {\sl}  
  {\parindent}     
  {\bf}  
  {. }    
  { }    
  {}     
\theoremstyle{theorem}
\newtheorem{thm}{Theorem}[section]  
\newtheorem{cor}[thm]{Corollary}
\newtheorem{conj}[thm]{Conjecture}

\newtheorem{prop}[thm]{Proposition}

\newtheoremstyle{definition}
  {12pt}		  
  {0pt}  
  {}  
  {\parindent}     
  {\bf}  
  {. }    
  { }    
  {}     
\theoremstyle{definition}

\renewcommand{\proofname}{Proof}

\makeatletter
\renewenvironment{proof}[1][\proofname]{\par
  \pushQED{\qed}%
  \normalfont \partopsep=\z@skip \topsep=\z@skip
  \trivlist
  \item[\hskip\labelsep
        \scshape
    #1\@addpunct{.}]\ignorespaces
}{%
  \popQED\endtrivlist\@endpefalse
}
\makeatother

\makeatletter
\renewcommand*\@maketitle{%
  \normalfont\normalsize
  \@adminfootnotes
  \@mkboth{\@nx\shortauthors}{\@nx\shorttitle}%
  \global\topskip0\p@\relax 
  \@settitle
  \ifx\@empty\authors \else {\vskip 1em
\vtop{\centering\shortauthors\@@par}} \fi
  \ifx\@empty\@date \else {\vskip 1em \vtop{\centering\@date\@@par}}\fi 
  \ifx\@empty\@dedicatory
  \else
    \baselineskip18\p@
    \vtop{\centering{\footnotesize\itshape\@dedicatory\@@par}%
      \global\dimen@i\prevdepth}\prevdepth\dimen@i
  \fi
  \@setabstract
  \normalsize
  \if@titlepage
    \newpage
  \else
    \dimen@34\p@ \advance\dimen@-\baselineskip
    \vskip\dimen@\relax
  \fi
} 
\renewcommand*\@adminfootnotes{%
  \let\@makefnmark\relax  \let\@thefnmark\relax
  \ifx\@empty\@subjclass\else \@footnotetext{\@setsubjclass}\fi
  \ifx\@empty\@keywords\else \@footnotetext{\@setkeywords}\fi
  \ifx\@empty\thankses\else \@footnotetext{%
    \def\par{\let\par\@par}\@setthanks}%
  \fi
\thispagestyle{titlepage}
}
\makeatother

\pgfsetlayers{nodelayer,edgelayer} 
\usepackage{graphicx}

\begin{document}
\title[]{\large A character theoretic\\ formula for base size II}

\author{Coen del Valle}
\address{School of Mathematics and Statistics,
The Open University, Milton Keynes, UK
}
\email{Coen.del-Valle@open.ac.uk}

\thanks{The author thanks Martin Liebeck for a helpful early conversation, and Peter Cameron for his comments on a draft. Research of Coen del Valle is supported by the Natural Sciences and Engineering Research Council of Canada [funding reference number PGSD-577816-2023], and by the Engineering and Physical Sciences Research Council
[grant number EP/Z534742/1].}

\date{\today}

\begin{abstract}
A base for a permutation group $G$ acting on a set $\Omega$ is a sequence $\mathcal{B}$ of points of $\Omega$ such that the pointwise stabiliser $G_{\mathcal{B}}$ is trivial. The base size of $G$ is the size of a smallest base for $G$. Extending the results of a recent paper of the author, we prove a 2013 conjecture of Fritzsche, K\"ulshammer, and Reiche. Moreover, we generalise this conjecture and derive an alternative character theoretic formula for the base size of a certain class of permutation groups. As a consequence of our work, a third formula for the base size of the symmetric group of degree $n$ acting on the subsets of $\{1,2,\dots, n\}$ is obtained.
\end{abstract}

\maketitle
\hrule

\bigskip

\section{Introduction}
A \emph{base} for a permutation group $G$ acting on a finite set $\Omega$ is a sequence $\mathcal{B}$ of points of $\Omega$ with trivial pointwise stabiliser $G_\mathcal{B}$. The size $b(G)$ of a smallest base for $G$ is called the \emph{base size} of $G$. In 1992, Blaha ~\cite{blaha} showed that the problem of finding a minimum base for an arbitrary group $G$ is NP-hard. Despite this, much work has been done towards determining the base size of certain families of groups, especially primitive groups; we recommend the paper of Mar\'oti~\cite{Maroti23} for a survey of many significant results in the area.

Let $n$ and $k$ be positive integers with $n\geq 2k$ and let $\S_{n,k}$ be the symmetric group $\S_n$ acting on the $k$-element subsets of $\{1,2,\dots,n\}$. The group $\S_{n,k}$ is primitive provided that $n>2k$. In 2013, the base size of $\mathrm{S}_{n,k}$ was determined by Fritzsche, K\"ulshammer, and Reiche~\cite[Lemma 3.3]{fkr} in a beautiful result which appears to have either gone unnoticed, or forgotten in the literature --- one purpose of this paper is to rectify this oversight. Over 10 years later, two independent papers~\cite{basesn,MeSp} were published, once again determining formulae for $b(\mathrm{S}_{n,k})$; the formula in~\cite[Theorem 1.1]{basesn} is the same as that which appears in~\cite{fkr}, but the formula of Mecenero and Spiga~\cite[Theorem 1.1]{MeSp} takes a remarkably different form. During the review process, it was noted that the result~\cite[Theorem 1.1]{MeSp} has an entirely character theoretic interpretation. In particular, after some straightforward algebraic manipulation one derives that if $\mathrm{sgn}$ is the sign character of $\S_n$ and $\chi$ is the permutation character of $\S_{n,k}$, then $$b(\S_{n,k})=\min\{l\in \mathbb{N} : \langle \mathrm{sgn},\chi^l\rangle\ne0\}.$$

A recent paper of the author~\cite{char} exhibited an entirely algebraic proof of the above formula, and extended it to all groups admitting a \emph{base-controlling} homomorphism. We say that {${\phi:G\to\{1,-1\}}$} is \emph{base-controlling} if for every tuple $\mathcal{A}$ of points of $\Omega$, $\mathcal{A}$ is a base if and only if $\phi(G_{\mathcal{A}})=1$. Note that any base-controlling homomorphism of $G$ is an irreducible character of $G$. Define $\langle - , - \rangle$ to be the standard inner product of (complex-valued) class functions. That is, for class functions $\varphi_1,\varphi_2$ of $G$, $$\langle\varphi_1,\varphi_2\rangle=|G|^{-1}\sum_{g\in G}\varphi_1(g)\overline{\varphi_2(g)}.$$ One main result of~\cite[Theorem 1.2]{char} is that if $\chi$ is the permutation character of a group $G$ which admits a base-controlling homomorphism $\phi$, then \begin{equation}\label{charform1}
    b(G)=\min\{l\in\mathbb{N} :\langle \phi,\chi^l\rangle\ne0\}.
\end{equation}

The paper~\cite{fkr} concludes with a lovely conjecture, suggesting an alternative character theoretic formula for the base size of $\mathrm{S}_{n,k}$. Before stating their conjecture, we first give the necessary background.

%
Let $G$ be a finite permutation group with point stabiliser $H$. Define the \emph{K\"ulshammer graph} $\mathcal{K}(G,H)$ to have vertex set $\mathrm{Irr}(H)$ where $\alpha$ is adjacent to $\beta$ if and only if the induced characters $\alpha\ind$ and $\beta\ind$ have a common irreducible constituent. By~\cite[Corollary~6.3]{bkk}, the graph $\mathcal{K}(G,H)$ is connected; for any $\alpha,\beta\in\mathrm{Irr}(H)$, define $d(\alpha,\beta)$ to be the length of a shortest path from $\alpha$ to $\beta$ in $\mathcal{K}(G,H)$. Define the \emph{diameter} of $\mathcal{K}(G,H)$ to be $\mathrm{Diam}(\mathcal{K}(G,H))=\max_{\alpha,\beta\in\mathrm{Irr}(H)}d(\alpha,\beta).$ 

The K\"ulshammer graph has been studied several times in the past (see e.g.~\cite{bkk,burn,fkr}), primarily for its utility in determining the \emph{depth} of a subgroup $H$ in a group $G$. The definition of this notion of depth is beyond the scope of this paper but we point the interested reader to \cite{bdk,bkk,burn,fkr} for further discussion and to~\cite[Proposition 2.5]{burn} for a nice collection of facts relating the combinatorial information of the K\"ulshammer graph to depth.

We are now ready to present the conjecture of Fritzsche, K\"ulshammer, and Reiche~\cite{fkr}; throughout this paper we use $1_H$ to denote the trivial character of a group $H$, and use $\hspace{0.07cm}\ind$ and $\hspace{0.07cm}\res$ to denote the induction and restriction of characters, respectively.

\begin{conj}[\cite{fkr}]\label{mainconj}
    Let $n\geq 2k$, let $G=\mathrm{S}_{n,k}$, and let $H$ be a point stabiliser of $G$. Then $$b(G)=\mathrm{Diam}(\mathcal{K}(G,H))+1=d(1_H,\mathrm{sgn}\res)+1.$$
\end{conj}

In this paper we settle Conjecture~\ref{mainconj}; our main result is the following, which gives us another character theoretic formula for base size, distinct from that of~\cite{char}.

\begin{thm}\label{mainthm}
    Let $G$ be a finite permutation group with point stabiliser $H$. Suppose that $G$ admits a base-controlling homomorphism. Then $$b(G)=\mathrm{Diam}(\mathcal{K}(G,H))+1=d(1_H,\phi\res)+1.$$
\end{thm}

Since the sign character is base-controlling for $G=\mathrm{S}_{n,k}$ (see~\cite[Section 3]{char} for details) we immediately deduce the following.

\begin{cor}
    The Fritzsche--K\"ulshammer--Reiche Conjecture is true.
\end{cor}
Thus, we obtain a third formula for the base size of $\mathrm{S}_{n,k}$.

The structure of this paper is straightforward. In Section~\ref{sec2} we present a proof of Theorem~\ref{mainthm}, and in Section~\ref{sec3} we work through a couple of easy examples. Throughout, we assume familiarity with some basic concepts in character theory, and refer the reader to~\cite{MR450380} for the necessary background.

\section{Proof of Theorem~\ref{mainthm}}\label{sec2}
Throughout this section $G$ is a finite permutation group with base-controlling homomorphism $\phi$ and point stabiliser $H$, and $\chi=1_H\ind$ is the permutation character of $G$. Before stating our key proposition we remind the reader of a standard result of character theory which will be very useful for us: if $\alpha$ is any character of $G$, then \begin{equation}\label{fact}
\alpha\res\ind=\alpha\cdot\chi.
\end{equation}
The equality~\eqref{fact} is a special case of~\cite[Chapter 7.2, Remark (3)]{MR450380}.
\begin{prop}\label{pathprop}
Let $1_H,\alpha_1,\alpha_2,\dots,\alpha_m$ be a path in $\mathcal{K}(G,H)$. Then $\langle\chi^k,\alpha_k\ind\rangle\ne0$ for all $1\leq k\leq m$.
\end{prop}
\begin{proof}
    We prove the result by induction. The result holds for $k=1$: Indeed, since $1_H$ and $\alpha_1$ are joined by an edge it follows that $0\ne\langle 1_H\ind,{\alpha_1\ind}\rangle=\langle \chi^1,\alpha_1\ind\rangle,$ as desired. Assume the result holds for some $k\geq 1$. Then $$0\ne\langle \chi^k,\alpha_k\ind\rangle=\langle\chi^k\res,\alpha_k\rangle$$ by Frobenius reciprocity. That is, $\alpha_k$ is an irreducible constituent of $\chi^k\res$. Additionally, there is an edge between $\alpha_k$ and $\alpha_{k+1}$, whence $$0\ne\langle\alpha_k\ind,\alpha_{k+1}\ind\rangle=\langle\alpha_k,\alpha_{k+1}\ind\res\rangle.$$ It follows that $\alpha_k$ is a common irreducible constituent of both $\alpha_{k+1}\ind\res$ and $\chi^k\res$, thus $$0\ne\langle\chi^k\res,\alpha_{k+1}\ind\res\rangle=\langle\chi^k\res\ind,\alpha_{k+1}\ind\rangle=\langle\chi^{k+1},\alpha_{k+1}\ind\rangle,$$ where the final equality is~\eqref{fact}, hence the result.
\end{proof}
We are now ready to prove Theorem~\ref{mainthm}.
\begin{proof}[Proof of Theorem~\ref{mainthm}]
In~\cite[Corollary 1.4]{fkr}, it is shown that $$b(G)\geq \mathrm{Diam}(\mathcal{K}(G,H))+1,$$ and moreover, it is clear that $d:=d(1_H,\phi\res)\leq\mathrm{Diam}(\mathcal{K}(G,H)),$ so it suffices to show that $b(G)\leq d+1$.

Since $\phi$ is linear, it follows that $\phi\res\in\mathrm{Irr}(H)$. Thus, by Proposition~\ref{pathprop}, $$0\ne\langle \chi^{d},\phi\res\ind\rangle=\langle \chi^{d}\res,\phi\res\rangle=\langle \chi^{d}\res\ind,\phi\rangle=\langle \chi^{d+1},\phi\rangle,$$ by repeated applications of Frobenius reciprocity and~\eqref{fact}. Finally, since $\langle \chi^{d+1},\phi\rangle\ne 0$, we deduce from~\eqref{charform1} that $b(G)\leq d+1$, as was to be shown.
\end{proof}

\section{Other examples}\label{sec3}
In this section we present a couple of easy worked examples, demonstrating the utility of Theorem~\ref{mainthm} beyond the groups $\mathrm{S}_{n,k}$. 

First we let $G=\mathrm{PGL}_2(7)$ with its natural action on the 1-dimensional subspaces of the natural module $\mathrm{GF}(7)^2$. Then $G$ has point stabiliser $H=7\cn6$, and $G$ admits a base controlling homomorphism $\phi$ (see~\cite[Section 3]{char} for details). The stabiliser $H$ has 7 irreducibles characters which we label as $1_H,\phi\res,\alpha_1,\alpha_2,\alpha_3,\alpha_4,\alpha_5$. Consider the graph $\mathcal{K}(G,H)$, which is depicted in Figure~\ref{fig:KGH}. It is clear that the graph has diameter 2, and moreover, the distance between $1_H$ and $\phi\res$ is indeed $2$. This agrees with what we expect since $G$ is sharply 3-transitive and thus has base size 3.
    \begin{figure}[h]
        \centering
        \includegraphics[width=0.4\linewidth]{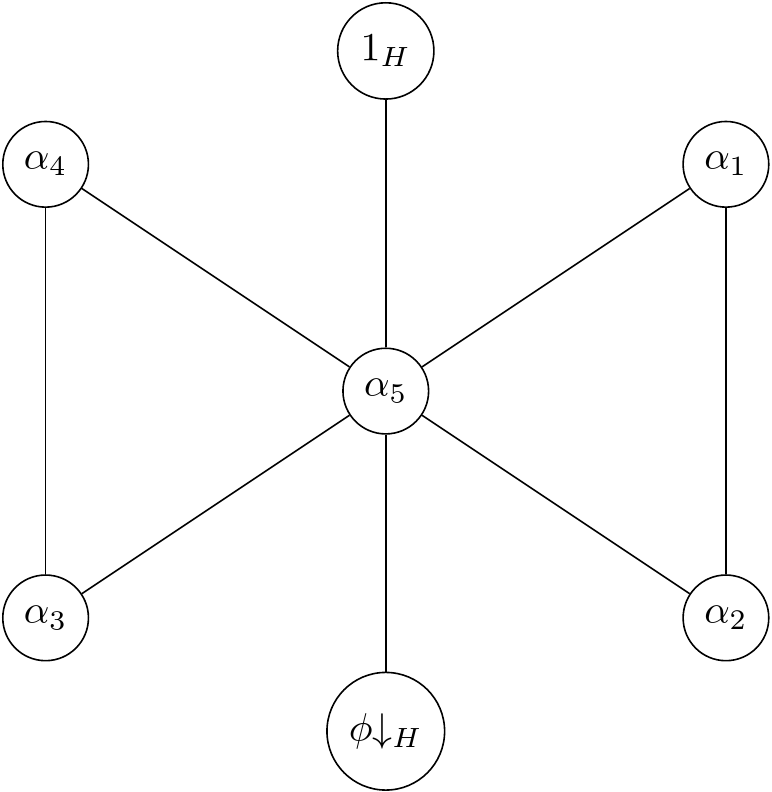}
        \caption{The graph $\mathcal{K}(\mathrm{PGL}_2(7),7\cn6)$.}
        \label{fig:KGH}
    \end{figure}

    We conclude the paper with a final example. Let $G=\mathrm{D}_{2n}$ be a dihedral group of degree $n\geq 2$ equipped with its natural action and point stabiliser $H$. Then $H$ has exactly two irreducible characters, and it is straightforward to check that the unique non-trivial character is the restriction to $H$ of a base-controlling homomorphism $\phi$ of $G$. Since $\mathcal{K}(G,H)$ is necessarily connected, we deduce that $\mathrm{Diam}(\mathcal{K}(G,H))+1=d(1_H,\phi\res)+1=2$, which agrees with the expected base size.
\bibliographystyle{plain}
\vspace{-0.3cm}
\bibliography{biblio}

\end{document}